\documentclass[12pt]{amsart}
\usepackage{amssymb,amscd,amsfonts,amsbsy,amsmath,amsthm,amsrefs}

\usepackage{indentfirst}

\usepackage{xcolor}

\usepackage{latexsym}
\usepackage{graphicx}

\usepackage[T1]{fontenc}

\usepackage[normalem]{ulem}
\usepackage{soul}
\usepackage{hyperref}
\hypersetup{
	colorlinks   = true,    
	urlcolor     = blue,    
	linkcolor    = blue,   
	citecolor    = red     
}

\catcode`@=11 \@addtoreset{equation}{section}
\renewcommand\theequation{\thesection.\@arabic\c@equation}
\catcode`@=12

\newcommand{\dd}{\,{\rm d}}

\newcommand\R{\mathbb{R}}
\renewcommand\P{{\mathbb{P}}}

\newcommand{\RR}{\R}

\definecolor{NavyBlue}{rgb}{0.00, 0.35, 0.50}
\definecolor{OliveGreen2}{cmyk}{0.64,0,0.95,0.40}

\newtheorem {Theorem}  {Theorem}

\numberwithin{Theorem}{section}
\newtheorem{Lemma}[Theorem]{Lemma}

\newtheorem{Corollary}[Theorem]{Corollary}
\theoremstyle{definition}

\theoremstyle{remark}
\newtheorem{Remark}[Theorem]{Remark}
\newtheorem{Example}[Theorem]{Example}

\begin{document}
	
	\title
	{Non-Algebraic Decay for Solutions to the  Navier--Stokes Equations}
	
	\author[L. Brandolese]{Lorenzo Brandolese}
	\address[L. Brandolese]{Université Claude Bernard Lyon 1, CNRS, Centrale Lyon, INSA Lyon, Université Jean Monnet, ICJ UMR5208,
		69622 Villeurbanne, France.}

	\author[M. Pageard]{Matthieu Pageard}
	\address[M. Pageard]{Université Claude Bernard Lyon 1, CNRS, Centrale Lyon, INSA Lyon, Université Jean Monnet, ICJ UMR5208,
		69622 Villeurbanne, France.}

	\author[C. F. Perusato]{Cilon F. Perusato}
	\address[C. F. Perusato]{Departamento de Matem\'atica. 
		Universidade Federal de Pernambuco, Recife, PE 50740-560. Brazil.}

	\thanks{Partially funded by CNPq grants N. ~ 200124/2024-2 and N.~310444/2022-5, CAPES-COFECUB grant N.~88887.879175/2023-00, ANR-25-CE40-4532 and CNRS program AMI allocations internationales.} 
	
	
	
	\date{\today}

	\begin{abstract}
		Around forty years ago, Michael Wiegner provided, in a seminal paper, sharp algebraic decay rates for solutions of the Navier--Stokes equations, showing   that these solutions behave asymptotically like the solutions of the heat equation with the same data as $t\to+\infty$, in the $L^2$-norm, up to some critical decay rate. 
		In the present paper, we close a  gap that appears in the conclusion of Wiegner's theorem in the 2D case, for solutions with non-algebraic decay rate.
	\end{abstract}
	
	\maketitle
	\definecolor{OliveGreen1}{rgb}{0,0.6,0}
	\definecolor{OliveGreen2}{cmyk}{0.64,0,0.95,0.40}

	\section{Introduction}
	This study is motivated by the recent paper~\cite{CPZ}, where it is proved that,
	in the class of Leray solutions of the unforced Navier--Stokes equations, the subclass of solutions with $L^2$-algebraic decay is negligible, in an appropriate topological sense.
	This leads us to study solutions with non-algebraic decay rates.
	
	The Navier--Stokes equations in the whole space $\R^d$ ($d=2,3$) read
	\begin{equation}\label{NS}
		\left\{		\begin{aligned}
			&\partial_t u+u\cdot \nabla u=\Delta u-\nabla p + f, \\
			&\nabla\cdot u=0, \\
			&u(x,0)=u_0(x),
		\end{aligned}
		\qquad x\in\R^d,\;t>0,
		\right.
	\end{equation}
	where $u=(u_1,\ldots,u_d)$ is the velocity field, and $p$ is the (scalar) pressure field.
	In this work, we will always assume $u_0\in L^2_\sigma(\R^d)$, the space
	of solenoidal vector fields with components in $L^2$, and
	that the given external force~$f=f(x,t)$ satisfies 
	$f\in L^{1}((0, \infty), L^{2}(\R^{d}))$.   	
	
	We are interested in the Leray solutions to system \eqref{NS}, namely the functions $u$ verifying
	$u\in C_{w}([0, \infty), L^{2}_{\sigma}(\R^{d})) \cap L^{2}((0,\infty), \dot{H}^{1}(\R^{d})) $
	with 
	$u(\cdot,0)=u_0$,
	satisfying \eqref{NS} in the weak sense, as well as
	the energy estimate {in its strong form}:
	\begin{equation}\label{eqn_energy}
		\|u(t)\|^2+	2\int_{s}^{t} \| \nabla u(\tau)\|^2\dd\tau
		\leq 
		\|u(s)\|^2 + 2\,\int_{s}^t\!\!\int_{\R^{d}}|u(x,\tau)\cdot f(x,\tau)| \dd x  \dd\tau, 
	\end{equation}
	for $ s = 0 $ and almost all $ s > 0 $, and every $ t \geq s $ (here, $\|\cdot\|$ denotes the $L^2$-norm). 
	It is well known (see \cite{Wie87}) that such solutions satisfy,
	for some constant $C>0$,
	\begin{equation}\label{SEI}
		\|u(t)\|^2+	2\int_{s}^{t} \| \nabla u(\tau)\|^2\dd\tau
		\leq 
		\|u(s)\|^2 + C\,\int_{s}^t \|f(\tau)\| \dd\tau,
	\end{equation}
	for $ s = 0 $ and almost all $ s > 0 $, and all $ t \geq s $. 
	When $d=2$, the Leray solutions are unique and inequality~\eqref{SEI} is known to hold for all~$0\le s\le t$.
	
	We now provide a quick overview of the literature concerning the asymptotic behavior and decay rates of the energy $\|u(t)\|^2$ for solutions to the Navier--Stokes equations. This problem dates back to Leray, who in his pioneering paper \cite{Leray1934} left open the question of whether the weak solutions he constructed, in the absence of external forcing, vanish to zero in $L^2$ as $t \to +\infty$. 
	This question was positively answered by Kato \cite{Kato1984} and Masuda \cite{Masuda1984} only fifty years later. Subsequently, Schonbek \cite{Sch85} obtained the first explicit decay rates by using her well-known {\it Fourier-Splitting} technique, under an additional  $L^p$-condition, $1\le p<2$, on the initial data.
	Optimal algebraic decay rates for the $L^2$-norm of Leray solutions were obtained by Wiegner \cite{Wie87}, with or without forcing. Moreover, the conditions required on~$u_0$ to achieve an algebraic decay prescribed in~\cite{Wie87} are also optimal, as later shown in~\cite{Ska}.
	Energy dissipation to zero in the presence of forcing satisfying
	conditions more general than
	$f\in L^1((0,\infty),L^2(\R^d))$
	is obtained in~\cite{ORS97}.
	The large time behavior of these solutions was subsequently extensively investigated in various directions, both for the Navier--Stokes equations and many of their generalizations. We refer to~\cite{BraS} for a detailed review of this topic and to~\cites{Fuji2025,Takeuchi1,Takeuchi2} for a sample of the most recent results.
	
	However, energy dissipation at algebraic rates is far from being the generic scenario; rather, it represents an exceptional situation. 
	To illustrate this fact, let us consider the class $\mathcal{L}$ of all Leray solutions to the unforced Navier--Stokes equations. We can endow $\mathcal{L}$ with the initial topology induced by the map 
	$I\colon \mathcal{L}\to  L^2_\sigma$, where $I$ is the map 
	$ u\mapsto  u_0$, 
	that associates, to a solution $ u$, 
	the corresponding initial datum. From this, in \cite{CPZ} the authors established the following result. 	
	
	\begin{Theorem}[L.\,Brandolese, C.\,Perusato,\, and P.\,Zingano,\,2024\,\cite{CPZ}]
		\label{th:genler}
		With the above topology, $\mathcal{L}$ is a Baire space 
		and the
		set of unforced Leray solutions with algebraic decay 
		is meager in~$\mathcal{L}$.
	\end{Theorem}
	
	Motivated by the above result, the purpose of this work is to extend  Wiegner's theorem \cite{Wie87}, to cover a considerably {\it larger class of decay profiles}, encompassing the Leray solutions with non-algebraic decay.
	As we will see, such an extension is especially relevant in the two-dimensional case, as the classical Wiegner theorem provides little
	information in this case. As a byproduct of our analysis, we will be able to  close a gap in the 2D case.
	
	To provide a deeper background for the present study, let us recall the already mentioned Wiegner result~\cite{Wie87}. 
	For the external force, Wiegner's assumptions are the following:
	
	for all $ t > 0 $,
	\begin{equation}\label{eqn_2_f}
		\begin{cases}
			f \in L^{1}((0, \infty), L^{2}(\R^{d})), \\
			\| f(t) \| \lesssim (1 + t)^{-\alpha - 1}, \\
			\| f(t) \|_{L^{d}(\R^{d})} \lesssim t^{-\alpha - (d+2)/4}.
		\end{cases}
	\end{equation}
	We also consider the solution 
	$v(t) \in C([0,\infty),L^2_\sigma(\R^d))$ of the associated Stokes flow, that in the whole space agrees with the solution of the heat equation:
	\begin{equation}\label{stokes-flow}
		\partial_t v = \Delta v + f, \quad v(\cdot,0) = u_0.
	\end{equation}
	
	\begin{Theorem}[M.\,Wiegner,\,1987\,\cite{Wie87}]\label{Wiegner_thm}
		Let $d\ge2$, $u_0\in L^2_\sigma(\R^d)$, and 
		$ f 
		\in L^{1}(( 0, \infty), L^{2}(\R^{d}))$.
		Then, $\|u(t)\|\to0$ as $t\to+\infty$. 
		Moreover, if $f$ satisfies $\eqref{eqn_2_f}$ for some  
		$ 0 \leq \alpha \leq (d+2)/4 $ and if
		\begin{equation}
			\|v(t)\|
			=    O(1 + t)^{-\alpha}, \quad 
			\text{for some $0 \leq \alpha\leq (d+2)/4$}, 
		\end{equation}
		then
		$\|\mbox{$u$}(t)\|
		=O(1 + t)^{-\alpha}$.
		In addition,
		\begin{equation}\label{concl:Wie}
			\|u(t)-v(t)	\|
			=
			\left\{\,
			\begin{array}{ll}
				o(1 + \hspace{+0.020cm} t)^{-(d-2)/4} & 
				\mbox{ if } \alpha = 0, \\
				O(1 +t)^{-2\alpha-(d-2)/4} & 
				\mbox{ if } 0 < \alpha < \frac{1}{2}, \\
				O(1 + t)^{-(d + 2)/4}    
				\log   (2 + t) & 
				\mbox{ if } \alpha = \frac{1}{2}, \\
				O(1 + t)^{-(d + 2)/4} & 
				\mbox{ if } \frac{1}{2} < \alpha \leq \frac{d+2}{4}.
			\end{array}
			\right.
		\end{equation}
	\end{Theorem}
	Notice that, when $0\le \alpha<(d+2)/4$ and $\|v(t)\|\approx t^{-\alpha}$,
	we have $\|u(t)-v(t)\|=o(t^{-\alpha})$ as $t\to+\infty$ by~\eqref{concl:Wie}. 
	Thus, Theorem~\ref{Wiegner_thm} \emph{essentially} states that, up to the critical decay rate 
	$t^{-(d+2)/4}$, the solution $u$ behaves asymptotically, in the $L^2$-norm, like the solution~$v$ of the corresponding linear problem, at the first order.
	However, a closer look at the theorem reveals that there is a gap in this interpretation.
	
	Indeed, let us consider the particular case $d=2$, $f\equiv0$ and $\alpha=0$: in this case, Theorem~\ref{Wiegner_thm} asserts that $\|u(t)-v(t)\|=o(1)$, as $t\to+\infty$.
	But for the linear part, it is well known that $\|v(t)\|=\|e^{t\Delta}u_0\|=o(1)$ as $t\to+\infty$, so we cannot conclude in this case that $u(t)$ behaves asymptotically, in the $L^2$-norm, like $v(t)$.
	For example, if $u_0$ is such that $\|e^{t\Delta}u_0\|$ decays at a logarithmic rate, then one might expect that $\|u(t)\|$ itself would decay at the same  logarithmic rate, but such a conclusion is \emph{not} a consequence of Wiegner's theorem in 2D; yet, it will be a consequence of our results, namely Assertion~i) of Theorem~\ref{th:intro} below. 
	
	The above difficulty is specific to the 2D case:
	in three or higher spatial dimensions, 
	the information about the decay of $\|u(t)-v(t)\|$ provided by~\eqref{concl:Wie} is more precise, and the above gap no longer appears. See Remark~\ref{3D1} for additional explanation on this point.
	For this reason, in the present paper we will only focus on the 2D case.

	\subsection*{Statement of the main result.}

	\begin{subequations}
		We collect here the different assumptions 
		on the external force that we will need in this paper.
		The (already mentioned) first condition that will always be required is 
		\begin{equation}\label{assu0:f} 
			f \in L^{1}((0, \infty), L^{2}(\R^{2})).
		\end{equation}   	
		A second assumption, that will be of constant use, is the following:
		\begin{equation}
			\label{assu1:f}
			t\mapsto \sqrt{t}\|f(t)\|\in L^2([T,+\infty)),
			\quad\text{for some $T\ge0$}.
		\end{equation}
		Our last condition, though not systematically required, but useful
		to strengthen our conclusions, is 
		\begin{equation}
			\label{assu2:f}
			\|f(t)\|=O(t^{-1})\Phi(t), \quad\text{as $t\to +\infty$},
		\end{equation}
	\end{subequations}
	for a function~$\Phi$, defined in $[0,+\infty)$, 
	such that $\lim_{t\to+\infty}\Phi(t)=0$. 
	This latter condition is more stringent than the previous one~\eqref{assu1:f}: indeed, if $f$ satisfies both~\eqref{assu0:f} and~\eqref{assu2:f},
	then condition~\eqref{assu1:f} follows by interpolation. Indeed, by ~\eqref{assu2:f} and the fact that $\Phi(t)\to0$, we have
	$t\|f(t)\|^2\lesssim \Phi(t)\|f(t)\|\lesssim \|f(t)\|$ for all large $t$,
	and the integrability condition ~\eqref{assu1:f} follows after integration thanks to ~\eqref{assu0:f}.
	We will prescribe below a few additional 
	conditions on~$\Phi$.

	Our next assumptions concern the large-time decay of $v$ in the $L^2$-norm.
	First of all, it is well known, and easily checked using the Fourier transform and the dominated convergence theorem, that under the sole condition~\eqref{assu0:f},  the solution of the heat equation~\eqref{stokes-flow} with $L^2$-initial data~$u_0$ satisfies
	\[
	\|v(t)\|\to0,\qquad\text{as $t\to+\infty$}.
	\] 
	We will make this fact more precise by \emph{assuming} that
	\begin{equation*}
		\|v(t)\| \lesssim \Phi(t),
	\end{equation*}
	with $\Phi$ decaying to zero
	``without fluctuations'', in a sense formalized by conditions
	~\eqref{Ass-Phi0}-\eqref{Ass-Phi1} below: 
	\begin{subequations}
		\begin{equation}
			\label{Ass-Phi0}
			\text{$\Phi$ is monotonically decreasing on $[0,\infty)$, with
				$\Phi(0+)>0$ and $\lim_{t\to+\infty}\Phi(t)=0$},
		\end{equation}
		and
		\begin{equation}
			\label{Ass-Phi1}
			\exists \kappa\ge1, \;\forall t>0,\quad
			\frac{1}{t}\int_0^t\Phi\le \kappa\Phi(t).
		\end{equation}
		Notice that functions satisfying~\eqref{Ass-Phi0}
		are necessarily locally (square-)integrable on $[0,+\infty)$,
		so the integral $\int_0^t\Phi$ is well-defined. 
		
		In the second assertion of our main result,
		we will require a stronger condition on $\Phi$, namely
		\begin{equation}
			\label{Ass-Phi2}
			\exists \kappa\ge1,  \; \forall t>0,\quad
			\frac{1}{t}\int_0^t\Phi(s)^2\dd s\le \kappa\Phi(t)^2.
		\end{equation}
	\end{subequations}
	By the Cauchy-Schwarz inequality, one easily checks that~\eqref{Ass-Phi2} implies~\eqref{Ass-Phi1}.
	Typical examples of functions satisfying~\eqref{Ass-Phi0}-\eqref{Ass-Phi1}
	are $\Phi(t)=(1+t)^{-a}\ln(e+t)^{-b}$, with $0\le a<1$ and $b\in\R$	(and $b>0$ if $a=0$).
	Typical examples of functions satisfying~\eqref{Ass-Phi0}-\eqref{Ass-Phi2}
	are the same functions as before, but with the smaller range $0\le a<1/2$.
	Slower decaying functions, obtained~\emph{e.g.} by taking negative powers of iterated logarithms, such as $\Phi(t)=(\log(\log(2e+t)))^{-b}$ with $b>0$, are also possible and satisfy the most stringent	version of our conditions.
	Let us also observe that conditions~\eqref{Ass-Phi0}-\eqref{Ass-Phi1} exclude $\Phi\in L^1(\R^+)$ (see~\eqref{infi} below), whereas conditions~\eqref{Ass-Phi0}-\eqref{Ass-Phi2} exclude $\Phi\in L^2(\R^+)$.
	In particular, cases~ii) and iii) in Theorem~\ref{th:intro} below are mutually exclusive.

	\newpage
	
	We now state our main result.
	
	\begin{Theorem}
		\label{th:intro}
		Let  $u_0\in L^2_\sigma(\R^2)$ and $f$ satisfy~
		\eqref{assu0:f}-\eqref{assu1:f}. Let $u$ be the Leray solution
		of the Navier--Stokes equations \eqref{NS} with $d=2$, and $v$ be the solution of the heat equation~\eqref{stokes-flow}, starting from~$u_0$.
		Suppose that
		\[
		\|v(t)\|=O(\Phi(t)),\quad\text{as $t\to+\infty$}.
		\]
		with $\Phi$ satisfying conditions~\eqref{Ass-Phi0}-\eqref{Ass-Phi1}.
		
		Then, the following properties hold:
		
		\begin{enumerate}
			\item[i)]
			We have
			\begin{equation}
				\label{concl-f}
				\|u(t)\|=O(\Phi(t))
				\quad\text{ and }\quad
				\|u(t)-v(t)\|=o(\Phi(t)), \quad\text{as $t \to +\infty$}.	
			\end{equation}
			\item[ii)]
			Under the more stringent condition~\eqref{Ass-Phi2}
			on~$\Phi$ (which excludes $\Phi\in L^2(\R^+))$
			and~\eqref{assu2:f} on~$f$, we have
			\[
			\|u(t)-v(t)\|=O(\Phi(t)^2), \quad\text{as $t \to +\infty$}.	
			\]
			\item[iii)]
			Under the alternative condition $\Phi\in L^2(\R^+)$ (which excludes~\eqref{Ass-Phi2}) and condition~\eqref{assu2:f} on~$f$,
			we have
			\[
			\|u(t)-v(t)\|=O(t^{-1}), \quad\text{as $t \to +\infty$}.	
			\]
		\end{enumerate}
	\end{Theorem}

	As an illustration of our result,
	let us consider $u_0$ and $f$ as in Example~\ref{Ex:1} below.
	We have $\|v(t)\|\approx\Phi(t)$, where 
	$\Phi(t):=(\ln(e+t))^{-1/2}$.
	As $\|v(t)\|\lesssim\Phi(t)$ and $\Phi$ satisfies conditions~\eqref{Ass-Phi0}-\eqref{Ass-Phi2}, Conclusion~ii) of Theorem~\ref{th:intro} applies, and we obtain $\|u(t)-v(t)\|\lesssim\Phi(t)^2$.
	But since we also have $\Phi(t)\lesssim\|v(t)\|$, we deduce that $u(t)$  behaves asymptotically like $v(t)$ in the $L^2$-norm: we have $\|u(t)-v(t)\|=o(\|v(t)\|)$ as $t\to+\infty$.
	In particular, $\|u(t)\|\approx\|v(t)\|$.
	Notice that the application of Theorem~\ref{Wiegner_thm} in the 2D case
	would only give the weaker result $\|u(t)\|\to0$ as $t\to+\infty$, without a priori excluding that
	$\|u(t)\|$
	may decay slower or faster than $\|v(t)\|$.

	\begin{Remark}
		We point out that~\eqref{assu1:f} is purely an $L^2$-integrability condition on $t^{1/2} \| f(t) \|$, with no reference to algebraic decay. In particular,
		Condition~\eqref{assu1:f} is considerably more general than Wiegner's condition~\eqref{eqn_2_f} on the external force, even when $\alpha=0$.
		Indeed, condition~\eqref{eqn_2_f} implies $\|f(t)\|\lesssim (1+t)^{-1}$, so $t\|f(t)\|^2\lesssim\|f(t)\|$. But~\eqref{eqn_2_f} also implies $f \in L^{1}((0, \infty), L^{2}(\R^{2}))$, and we get~\eqref{assu1:f}.
	\end{Remark}

	\begin{Remark}
		\label{rem:come}
		Conclusion~iii) allows to achieve Wiegner's optimal 2D-decay rate 
		$O(t^{-(d+2)/4})=O(t^{-1})$ for $\|u(t)-v(t)\|$,
		which is the same as in Theorem~\ref{Wiegner_thm}, but under weaker conditions.
		It is easy to prescribe conditions on $u_0$ and $f$ implying
		that $v\in L^2(\R^+L^2(\R^2))$, which in turn ensures the required condition $\Phi\in L^2(\R^+)$ for the application of  Conclusion~{iii)}:
		indeed, $\|v\|_{L^2(\R^+,L^2(\R^2))}\le \|u_0\|_{\dot H^{-1}(\R^2)}+\|f\|_{L^2(\R^+,\dot H^{-2}(\R^2))}$ (see~\cite{BahCD11}*{Lemma 4.10}).
	\end{Remark}

	\begin{Remark}[The higher dimensional case]\label{3D1}    
		The main interest of Conclusions i)-ii) is that the function $\Phi$
		can be chosen to decay arbitrarily slowly as $t\to+\infty$.
		In dimension $d\ge3$, the results of the previous
		theorems can be essentially covered, at least for slowly decaying
		profiles $\Phi$, as a direct consequence of Wiegner's theorem.
		The present observation does not apply in the 2D case, because we do not have any information on the decay rate of $\|u(t)-v(t)\|$ from Wiegner's Theorem \ref{Wiegner_thm} in 2D when $\alpha=0$. 
		More precisely, when $d\ge3$, under the sole assumption~$u_0\in L^2_\sigma(\R^d)$
		(assuming here $f\equiv0$ for simplicity), one has the following key property:	
		\begin{equation}\label{3D2}
			\| u(t) - e^{t\Delta}u_0 \|_{L^{2}(\R^d)}
			= o( t^{-(d-2)/4}),
		\end{equation}
		by~\eqref{concl:Wie}.
		In particular, if we have 
		\begin{equation}\label{3D3}
			\|e^{t\Delta}u_0 \|_{L^2(\R^d)}
			\leq \Phi(t)\qquad (t \gg 1),
		\end{equation}
		then, as soon as $\Phi$ satisfies a condition of mild decay, such as~\emph{e.g.} 
		$\liminf_{t\,\rightarrow\,\infty} t^{-(d-2)/4}\Phi(t)> 0$, we will automatically have 
		\begin{equation}\label{3D4}
			\| u(t) \|_{L^2(\R^d)}
			=	O(\Phi(t))
			\qquad\text{and}\qquad
			\| u(t) - e^{t\Delta}u_0    
			\|_{L^2(\R^d)}
			= o(\Phi(t)),
		\end{equation}
		which agrees with Conclusion~i) of Theorem~\ref{th:intro}.
		Similarly,
		when $d\ge3$,
		it readily follows by Wiegner's theorem that
		\begin{equation*}
			\| u(t) \|_{L^2(\R^d)}= O(\Phi(t))
			\qquad\text{and}\qquad
			\| u(t) - e^{t\Delta}u_0\|_{L^2(\R^d)}
			= o(\Phi(t)^2),
		\end{equation*}
		by only assuming the slightly more stringent slow decay condition
		$\liminf_{t\,\rightarrow\,\infty}  
		t^{-(d-2)/4}\Phi(t)^{2} > 0$. 
	\end{Remark}

	\paragraph{\bf Notation} 
	For a function 
	$w = (w_1, \ldots, w_d) $,
	writing $w \in L^{2}(\R^{d}) $
	means that
	$w_{i}\in L^{2}(\R^{d}) $
	for all $ 1 \leq i \leq d $,
	while $w\in L^{2}_{\sigma}(\R^{d}) $
	means that $w \in L^{2}(\R^{d}) $
	and $\mbox{div}\,w = 0 $ 
	in the distributional sense.
	The orthogonal projection from $L^2(\R^d)$ onto $L^2_\sigma(\R^d)$
	is denoted by $\P$. Classical results on the Riesz transform imply that
	this map extends as a bounded linear operator on $L^p(\R^d)$, for $1<p<+\infty$.  
	For brevity, we will often simply write 
	$ \| \cdot \|$ instead of $\| \cdot \|_{L^{2}(\R^d)}$.
	
	We use $o(1)$ to denote
	a function that vanishes at infinity,
	$ O(1) $ for a bounded function and, similarly:
	$ O(g(t)) = O(1)\hspace{+0.030cm}g(t)$ and 
	$ o(g(t))= o(1) g(t)$.
	The notation $A(t)\lesssim B(t)$ 
	means that there exists a constant~$c>0$, independent of~$t$, 
	such that $A(t)\le cB(t)$. The notation $A(t)\approx B(t)$ means that we have both 
	$A(t)\lesssim B(t)$ and $B(t)\lesssim A(t)$.
	
	We denote by $\mathcal{S}'(\R^d)$ 
	the space of tempered distributions.
	The Fourier transform of a function (or of a tempered distribution)~$f$
	is denoted by $\widehat f$.
	We denote by $e^{t\Delta}$ the heat semigroup.
	
	The integrals over the whole space will be denoted simply by 
	$\int$, instead of $\int_{\R^d}$, unless the explicit indication
	of $ \R^d $ is more convenient.	
	
	Finally, for two vectors $a,b\in\R^d$, we denote by $a\otimes b\in\mathrm{Mat}_{d\times d}(\R)$ the matrix given by $(a\otimes b)_{ij}=a_ib_j$. For a matrix $A=A(x)\in\mathrm{Mat}_{d\times d}(\R)$, we define its divergence by $(\mathrm{div}\,A)_i:=\sum_j\partial_jA_{ji}$.

	\section{Proof of the main theorems}
	
	Let us start by discussing some elementary consequences of~\eqref{Ass-Phi0}-\eqref{Ass-Phi1}.
	These conditions imply that, for all $t>0$,
	\begin{equation}
		\label{doubling}
		\Phi(t/2)\le \frac{2}{t}\int_0^{t/2}\Phi(s)\dd s
		\le \frac{2}{t}\int_0^{t}\Phi(s)\dd s
		\le  2\kappa\Phi(t).
	\end{equation}
	Moreover, for all $t\ge1$, 
	\begin{equation}
		\label{lower_Phi}
		t\Phi(t)\ge\frac{1}{\kappa}\int_0^t\Phi(s)\dd s
		\ge\frac{1}{\kappa}\int_0^1\Phi=:c^{-1}>0.
	\end{equation}
	Notice that the condition $\Phi(0)>0$ would not be sufficient to imply that
	$c^{-1}>0$, so the stronger condition $\Phi(0+)>0$ is needed in~\eqref{Ass-Phi0}.
	
	Another consequence of~\eqref{Ass-Phi0}-\eqref{Ass-Phi1} is that
	\begin{equation}
		\label{infi}
		\lim_{t\to+\infty} t\Phi(t)=+\infty.
	\end{equation}
	Indeed, otherwise there would exist a constant $C>0$ and
	a sequence $t_n\to+\infty$ such that
	$\int_0^{t_n} \Phi(s)\dd s\le \kappa t_n\Phi(t_n)\le C$, and by the monotone convergence theorem we would obtain $\Phi\in L^1(\R^+)$.
	But then $t\Phi(t)\ge \frac{1}{2\kappa}\int_0^\infty\Phi(s)\dd s$ for sufficiently large $t$, which in turn contradicts that $\Phi\in L^1(\R^+)$ when $\Phi\not\equiv0$. 
	
	The first ingredient is the following lemma.

	\begin{Lemma}\label{gradient_decay_f}
		Let  $u_0\in L^2_\sigma(\R^2)$ and $f$ satisfy~
		\eqref{assu0:f}-\eqref{assu1:f}. Let $u$ be the Leray solution
		of the Navier--Stokes equations \eqref{NS} with $d=2$, starting from~$u_0$.
		Then, 
		\begin{equation*}
			\| \nabla u (t) \| = o(t^{-1/2}), \quad \text{as} \quad t \to +\infty. 		
		\end{equation*}
	\end{Lemma}
	\begin{proof}
		Let $\sigma>0$ and $t_0 > 0$. We multiply the first equation in the Navier--Stokes system \eqref{NS} by $2\,(s-t_0)^\sigma\,u$ and integrate the result over $\RR^2 \times [t_0,t]$. After a few computations, we get
		\begin{multline*}
			(t-t_0)^\sigma\,\|u(t)\|^2 
			+ 2 \int_{t_0}^t (s-t_0)^\sigma \| \nabla u (s) \|^2 \dd s\\
			\leq \sigma\,\int_{t_0}^{t} (s-t_0)^{\sigma-1} \|u(s) \|^2 \dd s
			+ 2 \int_{t_0}^t (s-t_0)^\sigma \|u(s)\|\,\| f(s) \| \dd s.
		\end{multline*}
		Let $\epsilon>0$.
		According to Theorem~\ref{Wiegner_thm}, $\|u(t)\|\to0$ as $t\to+\infty$ (recalling that assumption~\eqref{assu0:f} alone ensures this). From this fact and assumption~\eqref{assu0:f} on $f$, we have, for $t_0$ sufficiently large and all $t\ge t_0$, 
		\begin{equation*}
			\begin{split}
				(t-t_0)^\sigma\,\|u(t)\|^2 
				+ 2 \int_{t_0}^t (s-t_0)^\sigma \| \nabla u (s) \|^2 \dd s
				&\leq  
				\epsilon (t-t_0)^\sigma  + 2\,\sqrt{\epsilon}\,\int_{t_0}^t 
				(s-t_0)^\sigma \,\| f(s) \| \dd s\\
				&\leq 2 \epsilon (t-t_0)^\sigma.
			\end{split}
		\end{equation*}
		So in particular,
		\begin{equation}
			\label{integ:1}
			\int_{t_0}^t (s-t_0)^\sigma \| \nabla u (s) \|^2 \dd s
			\le \epsilon (t-t_0)^\sigma.
		\end{equation}
		Let us denote by $\omega:=\mbox{curl\,} u=\partial_1u_2-\partial_2u_1$ the vorticity of $u$, which satisfies the equation
		\begin{equation}\label{vorticity}
			\partial_t \omega + u \cdot \nabla \omega = \Delta \omega+\mbox{curl\,}f. 
		\end{equation}
		Multiplying equation~\eqref{vorticity} by $2\,(s-t_0)^{\sigma +1}\omega$,
		we obtain, after space-time integration,
		\begin{multline*}
			(t-t_0)^{1 + \sigma} \,\|\omega(t)\|^2 
			+ 2 \int_{t_0}^t (s-t_0)^{1+\sigma} \| \nabla \omega (s) \|^2 \dd s\\
			\begin{aligned}
				&=(1+\sigma\,)\int_{t_0}^{t} (s-t_0)^{\sigma} \|\omega(s) \|^2 \dd s 
				+ 2 \int_{t_0}^t (s-t_0)^{1+\sigma} \int_{\R^2} (\mbox{curl}\, f)\,\omega \dd x\dd s\\ 
				&\leq
				(1+\sigma\,)\int_{t_0}^{t} (s-t_0)^{\sigma} \|\omega(s) \|^2 \dd s 
				+2\int_{t_0}^{t} (s-t_0)^{1+\sigma} \| f(s) \|
				\, \|\nabla \omega (s)\| \dd s.
			\end{aligned}
		\end{multline*}
		Applying the Young inequality $2ab\le a^2+b^2$ to the last integrand, and absorbing one of the two terms in the integral on the left-hand side, we deduce
		\begin{multline*}
			(t-t_0)^{1 + \sigma}\,\|\omega(t)\|^2 
			+ \int_{t_0}^t (s-t_0)^{1+\sigma} \| \nabla \omega (s) \|^2 \dd s\\
			\leq 
			(1+\sigma\,)\int_{t_0}^{t} (s-t_0)^{\sigma} \|\omega(s) \|^2 \dd s 
			+\int_{t_0}^{t}	(s-t_0)^{1+\sigma} \| f(s) \|^2 \dd s.
		\end{multline*}
		But as $\nabla\cdot u=0$, we have $\|\omega(s)\|=\|\nabla u(s)\|$, and we can apply \eqref{integ:1} to gather
		\begin{equation*}
			(t-t_0)^{1 + \sigma}\,\|\omega(t)\|^2 + \int_{t_0}^t (s-t_0)^{1+\sigma} \| \nabla \omega (s) \|^2 \dd s
			\leq (1+\sigma)\,\epsilon\,(t-t_0)^{\sigma} 
			+ \int_{t_0}^{t}	(s-t_0)^{1+\sigma} \| f(s) \|^2 \dd s.
		\end{equation*}
		Now, in view of condition~\eqref{assu1:f}, we obtain, by possibly increasing $t_0$, 
		\begin{equation*}
			\begin{split}
				\int_{t_0}^{t}	(s-t_0)^{1+\sigma} \| f(s) \|^2 \dd s
				\leq (t-t_0)^\sigma\int_{t_0}^t s\|f(s)\|^2\dd s		
				\leq \epsilon\,(t-t_0)^\sigma. 
			\end{split}
		\end{equation*}
		Therefore, for some constant $C>0$ and all $t>t_0$,
		\begin{equation*}
			\|\omega(t)\|^2 = \|\nabla u (t)\|^2 \leq C\,\epsilon\,(t-t_0)^{-1}.
		\end{equation*}
		So $\|\nabla u(t)\| = o(t^{-1/2})$ as $t\to+\infty$ as desired. 
	\end{proof} 	
	
	\begin{Remark}
		\label{improved-grad-rmk}
		It is worth noting that the above argument improves upon the result presented in \cite{GuterresNichePerusatoZingano2023} in the following sense: there, in order to establish the simplest case of algebraic decay for $\nabla u$, the authors assumed stronger conditions (see~\emph{e.g.} \cite{GuterresNichePerusatoZingano2023}*{Theorem 1.9}), for instance involving the first derivative of $f$. 
		By contrast, our argument requires on~$f$ only to satisfy conditions~\eqref{assu0:f}-\eqref{assu1:f}.
		The previous argument can also be applied to obtain the following fact: 
		if, for some $\gamma>0$, we have $\|u(t) \|=O(t^{-\gamma})$
		and $\|f(t)\|=O(t^{-1-\gamma})$ as $t\to+\infty$, then
		\begin{equation}\label{improved-grad}
			\|\nabla u(t)\| = O(t^{-\gamma - 1/2} ),\,\, \text{as}\,\, t\to+\infty.  
		\end{equation}
		In order to obtain \eqref{improved-grad}, it suffices to multiply the first equation in the Navier--Stokes system \eqref{NS} by $2(s - t_0)^{\sigma + 2\gamma} u$ instead of $2(s - t_0)^{\sigma} u$, and then to repeat the same steps performed in the proof of Lemma \ref{gradient_decay_f} above.
	\end{Remark}

	With the preceding results at hand, we are ready to prove 
	the first conclusion of Theorem~\ref{th:intro}.
	
	\begin{proof}[Proof of Assertion~i) of Theorem~\ref{th:intro}]
		Let us write the Duhamel formula
		\[
		u(t) = v(t) - \int_0^te^{(t-s)\Delta}\P(u\cdot\nabla)u(s)\dd s.
		\]
		As $\nabla\cdot u=0$, we have $(u\cdot\nabla)u=\nabla\cdot(u\otimes u)$, hence we can decompose 
		\[
		u(t) = v(t) - \int_0^{t/2} e^{(t-s)\Delta}\P\nabla\cdot (u\otimes u)(s)\dd s - \int_{t/2}^t e^{(t-s)\Delta}\P(u\cdot\nabla)u(s)\dd s.
		\]
		Since $\P$ is bounded in $L^2$ and by the usual $L^1$-$L^2$ estimates for 
		$e^{t\Delta}$ and $e^{t\Delta}\nabla$, and using also that $\|v(t)\|\lesssim\Phi(t)$, we have, for some absolute constant $C_0>0$ and all $t>0$,
		\begin{equation}
			\label{L2du}
			\|u(t)\|\le C_0\left[\Phi(t)+\int_0^{t/2} (t-s)^{-1}\|u(s)\|^2\dd s
			+\int_{t/2}^t (t-s)^{-1/2}\|u(s)\|\,\|\nabla u(s)\|\dd s\right].
		\end{equation}
		
		Assumptions~\eqref{Ass-Phi0}-\eqref{Ass-Phi1} imply that $\Phi>0$. 
		Then, inspired by~\cite{ZinganoJMFM}, we can introduce, for all $t\ge0$, the quantities
		\[
		E(t):=\frac{\|u(t)\|}{\Phi(t)}
		\qquad	\text{and} 		\qquad
		E_{\text{max}}(t):=\sup_{s\in[0,t]}E(s).
		\]
		From~\eqref{L2du}, we have
		\begin{equation}
			\label{AA3}
			E(t)\le C_0 \big(1 + I_1(t) + I_2(t)\big),
		\end{equation}
		where
		\[
		I_1(t):=\frac{1}{\Phi(t)}\int_0^{t/2} (t-s)^{-1}\|u(s)\|^2\dd s
		\qquad	\text{and} 	\qquad
		I_2(t):=\frac{1}{\Phi(t)}\int_{t/2}^t (t-s)^{-1/2}\|u(s)\|\,\,\|\nabla u(s)\|\dd s.
		\]
		Notice that the integral defining $I_2(t)$ is finite for every $t>0$ because
		$\| \nabla u(s)\| = o(s^{-1/2})$ by Lemma~\ref{gradient_decay_f} and $\|u(s)\|$ is bounded. Let us first estimate $I_1$. From~\eqref{SEI}, we have, for all $t\ge0$,
		\begin{equation*}
			\|u(t)\|^2 \leq \| u_0 \|^2 + C\int_0^\infty \|f(\tau)\|\, \dd \tau\,=:\, M.
		\end{equation*}
		As previously observed, we have $\|u(t)\|\to0$ as $t\to+\infty$,
		hence there exists $t_1\ge1$ such that $\|u(s)\|\le (8C_0\kappa )^{-1}$ for all $s\ge t_1$, where $\kappa$ is given by \eqref{Ass-Phi1}.
		Then, for all $t\ge 2t_1$,
		\[
		\begin{split}
			I_1(t)
			&=\frac{1}{\Phi(t)}\biggl[\int_0^{t_1} (t-s)^{-1}\|u(s)\|^2\dd s
			+\int_{t_1}^{t/2} (t-s)^{-1}\|u(s)\|^2\dd s\biggr]\\
			&\le \frac{1}{\Phi(t)}\biggl[ M t_1(t-t_1)^{-1}
			+ \frac{2 (8C_0\kappa )^{-1} }{t}\int_{t_1}^{t/2} \|u(s)\|\dd s\biggr]\\
			&\le 
			\frac{1}{\Phi(t)}\biggl[ M t_1(t-t_1)^{-1}
			+ \frac{2 (8C_0\kappa )^{-1} }{t}\int_{t_1}^{t} E(s)\Phi(s)\dd s\biggr].
		\end{split}
		\]
		So, applying~\eqref{Ass-Phi1} and recalling~\eqref{lower_Phi},
		\[
		\begin{split}
			I_1(t)
			&\le \frac{2Mt_1}{t\Phi(t)}
			+\frac{2 (8C_0\kappa )^{-1} }{t\Phi(t)}\left(\int_{t_1}^{t} \Phi(s)\dd s\right) 					E_{\text{max}}(t)\\
			&\le \frac{2Mt_1}{t\Phi(t)}
			+2\kappa  (8C_0\kappa )^{-1} E_{\text{max}}(t)\\
			&\le  
			2cM t_1+
			\frac{E_{\text{max}}(t)}{4C_0}.
		\end{split}
		\]
		We now estimate $I_2$. By Lemma~\ref{gradient_decay_f},  $s^{1/2} \|\nabla u(s)\| \to 0$ as $s\to + \infty$.
		Therefore, 
		\begin{equation}
			\label{eps(t)}
			\varepsilon(t): =  \int_{t/2}^t (t-s)^{-1/2} \|\nabla u(s)\| \dd s \to0,
			\quad\text{as $t\to+\infty$.}
		\end{equation}
		We deduce that there exists $t_2\ge2t_1$ such that, for all $t\ge t_2$,
		\[
		\varepsilon(t)\le   (8C_0\kappa )^{-1}.
		\]
		Applying this and using ~\eqref{doubling}, we have, for all $t\ge t_2$, 
		\[
		\begin{split}
			I_2(t)
			&=
			\frac{1}{\Phi(t)}
			\int_{t/2}^t (t-s)^{-1/2}\|u(s)\|\,\,\|\nabla u(s)\|\dd s\\
			&\le
			\frac{1}{\Phi(t)}
			\int_{t/2}^t (t-s)^{-1/2}E(s)\Phi(s)\|\nabla u(s)\|\dd s\\
			&\le
			\frac{\Phi(t/2)}{\Phi(t)}E_{\text{max}}(t)
			\int_{t/2}^t (t-s)^{-1/2}\|\nabla u(s)\|\dd s\\
			&\le  2\kappa E_{\text{max}}(t)\varepsilon(t)\\
			&\le   \frac{E_{\text{max}}(t)}{4C_0}.
		\end{split}
		\]
		Going back to~\eqref{AA3}, we deduce that
		\[
		\begin{split}
			E(t)
			&\le
			C_0\bigl(1+
			2 c M t_1\bigr) +
			\frac{E_{\text{max}}(t)}{2},\qquad\forall\,t\ge t_2.
		\end{split}
		\]
		This implies
		\begin{equation}
			\label{dec:u}
			\|u(t)\|\le 
			2C_0\bigl(1+
			2 c M t_1\bigr)\Phi(t),
			\qquad\forall\,t\ge t_2.
		\end{equation}

		We now estimate the difference $\theta(t)=u(t)-v(t)$.
		For this purpose, observe that
		\[
		\theta(t)=-\int_0^t e^{(t-s)\Delta}\P (u\cdot\nabla) u(s)\dd s,
		\]
		hence
		\begin{equation}
			\label{duapre}
			\|\theta(t)\|
			\le C_0\int_0^{t/2} (t-s)^{-1}\|u(s)\|^2\dd s
			+C_0\int_{t/2}^t(t-s)^{-1/2}\|u(s)\|\,\|\nabla u(s)\|\dd s.
		\end{equation}
		Let us estimate the first integral. Using the decay property~\eqref{dec:u}, we obtain, for some constant $C>0$ and all $t>0$, 
		\begin{equation}
			\label{thetaine}
			\int_0^{t/2} (t-s)^{-1}\|u(s)\|^2\dd s \le 2t^{-1} \int_0^{t/2} \|u(s)\|^2\dd s
			\le Ct^{-1}\int_0^{t}\Phi(s)^2\dd s.
		\end{equation}
		Let $M':=\Phi(0)$. From assumption~\eqref{Ass-Phi0}, we have $\Phi(t)\to0$ as $t\to+\infty$, hence there exists $t_3\ge t_2$ such that $\Phi(t)\le1$ for all $t\ge t_3$. Define $\lambda(t):=t\Phi(t)$. Then, for all $t\ge t_3$, we have~$0<\lambda(t)\le t$, and 
		\[
		\begin{split}
			t^{-1}\int_0^{t}\Phi(s)^2\dd s
			&\le t^{-1}\int_0^{\lambda(t)}\Phi(s)^2\dd s
			+ t^{-1}\int_{\lambda(t)}^t\Phi(s)^2\dd s\\
			&\le M' t^{-1}\frac{\lambda(t)}{\lambda(t)}
			\int_0^{\lambda(t)}\Phi(s)\dd s
			+ t^{-1}\Phi(\lambda(t))\int_{\lambda(t)}^t\Phi(s)\dd s.\\
		\end{split}
		\]
		Applying the Cesàro condition~\eqref{Ass-Phi0} at the argument \(\lambda(t)\) to the first average
		and at \(t\) to the second (after extending the integral to \([0,t]\)) gives
		\[
		\begin{split}
			t^{-1}\int_0^{t}\Phi(s)^2\dd s \le M' \kappa t^{-1}\lambda(t)\Phi(\lambda(t))
			+ \kappa \Phi(t)\Phi(\lambda(t))
			\le \kappa(M'+1)\Phi(t)\Phi(t\Phi(t)). 
		\end{split}
		\]
		Let us now estimate the second integral in \eqref{duapre}. From the decay~\eqref{dec:u}, property~\eqref{doubling} and the definition \eqref{eps(t)} of $\varepsilon(t)$, we have, for all $t\ge t_3$,
		\begin{equation}
			\label{thetaine2}
			\int_{t/2}^t(t-s)^{-1/2}\|u(s)\|\,\|\nabla u(s)\|\dd s
			\lesssim \Phi(t)\int_{t/2}^t (t-s)^{-1/2}\|\nabla u(s)\|\dd s
			\lesssim \Phi(t)\varepsilon(t).
		\end{equation}
		We finally get
		\[
		\begin{split}
			\|\theta(t)\|
			&\lesssim \Phi(t)\Phi(t\Phi(t))
			+\Phi(t)\varepsilon(t)= o(\Phi(t)), 
			\qquad\text{as $t\to+\infty$},
		\end{split}
		\]
		because $t\Phi(t)\to+\infty$ by ~\eqref{infi}, and $\Phi(t)\to0$ as $t\to+\infty$.	
	\end{proof}

	Before proving the other assertions of Theorem~\ref{th:intro},
	we need another lemma, which is a refinement
	of~Lemma~\ref{gradient_decay_f}.
	
	\begin{Lemma}\mbox{}
		\label{lem:cil}
		Let  $u_0\in L^2_\sigma(\R^2)$ and $f$ satisfy~
		\eqref{assu0:f} and \eqref{assu2:f}. Let $u$ be the Leray solution
		of the Navier--Stokes equations \eqref{NS} with $d=2$, and $v$ be the solution of the heat equation~\eqref{stokes-flow}, starting from~$u_0$.
		Let $\Phi$ satisfy conditions~\eqref{Ass-Phi0}-\eqref{Ass-Phi1}, and be such that
		\[
		\|v(t)\| = O(\Phi(t)),\quad\text{as $t\to+\infty$}.
		\]
		Then, 
		\[
		\|\nabla u(t)\| = O(t^{-1/2}\Phi(t)),\quad\text{as $t\to+\infty$}.
		\]
	\end{Lemma}
	
	\begin{proof}
		We use once more the vorticity equation~\eqref{vorticity}:
		multiplying it
		by $ 2(s - \frac{t}{2}\,)\,\omega(y,s) $ 
		and integrating with respect to 
		$(y,s)\in \R^2 \times [t/2, t] $, we get
		\begin{equation*}
			\frac{t}{2}\,\|\omega(t)\|^2 + 2\int_{t/2}^t
			(s -{\textstyle\frac{t}{2}}\,) \|\nabla\omega(s)\|^2 \dd s = 
			\int_{t/2}^t  \|\omega(s)\|^2 \dd s 
			+2\int_{t/2}^t\!\int (s-\textstyle\frac t2)\mbox{curl\,}f(\cdot,s)				\omega(\cdot,s)\dd s.
		\end{equation*}
		From this, we deduce
		\begin{equation*}
			\frac{t}{2}\,\|\omega(t)\|^2 + 2\int_{t/2}^t
			(s -{\textstyle\frac{t}{2}}\,) \|\nabla\omega(s)\|^2 \dd s 
			\le 
			\int_{t/2}^t  \|\omega(s)\|^2 \dd s + 2\int_{t/2}^t (s-\textstyle\frac t2)\|f(s)\|\,\|\nabla \omega(s)\|\dd s.
		\end{equation*}
		Hence, applying the Young inequality in the last integrand, we gather
		\begin{equation*}
			\frac{t}{2}\,\|\omega(t)\|^2 
			+ \int_{t/2}^t	(s -{\textstyle\frac{t}{2}}\,) \|\nabla\omega(s)\|^2 \dd s 
			\le
			\int_{t/2}^t  \|\omega(s)\|^2 \dd s + \int_{t/2}^t (s-\textstyle\frac t2)\|f(s)\|^2\dd s.
		\end{equation*}
		Recalling that $\|\omega\| = \|\nabla u\|$ and using the energy inequality~\eqref{eqn_energy}, we obtain 
		\[
		\begin{split}
			\frac t2 \|\nabla u(t)\|^2
			&\le \int_{t/2}^t \|\nabla u(s)\|^2\dd s 
			+\int_{t/2}^t (s-\textstyle\frac t2)\|f(s)\|^2\dd s \\
			&\le \|u(t/2)\|^2 
			+2\int_{t/2}^t \|u(s)\|\,\|f(s)\|\dd s + \int_{t/2}^t (s-\textstyle\frac t2)\|f(s)\|^2\dd s.
		\end{split}
		\]
		Now, the conditions~\eqref{assu0:f} and~\eqref{assu2:f} on $f$ imply~\eqref{assu1:f} by interpolation. 
		Therefore, Conclusion~i) of~Theorem~\ref{th:intro} applies, and we deduce that 
		$\|u(t)\|\lesssim \Phi(t)$.
		From this, condition~\eqref{assu2:f} and \eqref{doubling}, we deduce that $\frac t2\|\nabla u(t)\|^2 \lesssim \Phi(t)^2$, for all $t>0$.
		This establishes Lemma~\ref{lem:cil}.
	\end{proof}

	The remaining parts of~Theorem~\ref{th:intro} are now easily established.

	\begin{proof}[Proof of Assertions~ii) and iii) of Theorem~\ref{th:intro}]	
		As observed in the proof of Lemma~\ref{lem:cil}, our assumptions imply the existence of some constant $C>0$ such that $\|u(t)\|\le C\Phi(t)$, for all $t\ge0$.
		
		We go back to inequalities~\eqref{duapre}-\eqref{thetaine2} for $\theta(t)=u(t)-v(t)$ and deduce
		\[
		\|\theta(t)\| \le Ct^{-1}\int_0^{t}\Phi(s)^2\dd s
		+ C\Phi(t)\int_{t/2}^t (t-s)^{-1/2}\|\nabla u(s)\|\dd s.
		\]
		The second term is treated by applying Lemma~\ref{lem:cil} and~\eqref{doubling}, that gives
		\begin{equation}
			\label{eq:border}
			\|\theta(t)\| \lesssim t^{-1}\int_0^{t}\Phi(s)^2\dd s
			+ \Phi(t)^2
			\lesssim
			t^{-1}\int_0^{t}\Phi(s)^2\dd s,
		\end{equation}
		where we used the monotonicity of $\Phi$ in the last inequality. Indeed, 
		$(t/2)\Phi(t)^2\le\int_{t/2}^{t}\Phi^2\le \int_{0}^{t}\Phi^2$.
		For estimating the right-hand side, we distinguish two cases. 
		Under the assumptions of Part~ii) of Theorem~\ref{th:intro},
		we can directly apply condition~\eqref{Ass-Phi2} and conclude that $\|\theta(t)\|\lesssim \Phi(t)^2$.
		Otherwise, under the assumptions of Part~iii) of Theorem~\ref{th:intro},
		we have $\Phi\in L^2(\R^+)$, and we readily get
		$\|\theta(t)\|\lesssim t^{-1}$.
	\end{proof}
	
	\begin{Remark}
		\label{rem:ans2}
		Let us discuss further the conditions that we put on~$\Phi$. 
		Estimate~\eqref{eq:border} holds assuming only~\eqref{Ass-Phi0}-\eqref{Ass-Phi1}.
		Condition~\eqref{Ass-Phi2} is mainly useful to state Theorem~\ref{th:intro} in a more attractive way, but such an assumption could be simply dropped by replacing Conclusion~ii) with the conclusion
		\begin{equation}
			\label{coclus}
			\|u(t)-v(t)\|=O\bigl(\textstyle\frac1t \int_0^t \Phi(s)^2\dd s\bigr)
			\qquad
			\text{as $t\to+\infty$}.
		\end{equation}

		In the case of algebraic decays, i.e., when $\Phi(t)=C(1+t)^{-\alpha}$,  Wiegner's theorem, for $d=2$, matches Conclusion~ii) of Theorem~\ref{th:intro} when $0<\alpha<1/2$ and Conclusion iii) when $1/2<\alpha$.
		On the other hand, the assertion of Wiegner's theorem in the borderline case $\alpha=1/2$ is not covered by Theorem~\ref{th:intro}
		the way it is stated. However,  we can recover Wiegner's conclusion also when $\alpha=1/2$ directly from \eqref{coclus}, that provides the same logarithmic correction in the decay rate as 
		in~\eqref{concl:Wie} .
		
		The Cesàro type condition~\eqref{Ass-Phi1} plays a deeper role in our arguments. There is no evidence that such a condition is \emph{necessary} for the validity of Conclusion~i) of Theorem~\ref{th:intro}, but we do not expect \eqref{Ass-Phi1} can be relaxed considerably. For example, a weaker condition like $t\Phi(t)\to+\infty$, instead of \eqref{Ass-Phi1}, does not seem to be enough. Indeed, this larger class would encompass decreasing functions~$\Phi$
		``fluctuating'' between two positive decreasing profiles $\Phi_1$ and $\Phi_2$, 
		in the sense that $\Phi_1(t)\le \Phi(t)\le \Phi_2(t)$ for all nonnegative $t$, with $\Phi(t_k)=\Phi_1(t_k)$ and $\Phi(t_k')=\Phi_2(t_k')$
		for some sequences $(t_k)$ and $(t_k')$ diverging  to $+\infty$. We could also take $\Phi_1$ and $\Phi_2$ 
		such that $\Phi_1(t)=O(\Phi_2(t)^2)$ as $t\to+\infty$.
		In such a situation, one can hope for a bound of the form $\|u(t)-v(t)\|=O(\Phi_2(t)^2)$, for example if $\Phi_2$ satisfies 
		\eqref{Ass-Phi2};  but because of the time average that
		genuinely appears when one writes $u(t)-v(t)$ using Duhamel formula, it seems impossible to get the stronger bound
		$\|u(t)-v(t)\|=o(\Phi_1(t))$ and so the bound $\|u(t)-v(t)\|=o(\Phi(t))$, as $t\to+\infty$.
		For this reason, Conclusion i) of Theorem~\ref{th:intro}, does not seem to be valid in such a more general setting.
		An explicit construction of an initial data such that the solution of the heat equation $v$ decays with similar fluctuations in the $L^2$-norm, is provided in~\cite[Example~3.1]{BrandoSIMA}.
	\end{Remark}
	
	\begin{Remark}
		Let us point out that, once the bound $\|u(t)\| = O(\Phi(t))$ is available, the estimate for $ \|\theta(t) \|$ in the proof of Assertion~i) of Theorem \ref{th:intro} relies solely on this bound and on the decay of $\|\nabla u (t) \|$ without invoking the hypothesis on $\|v(t)\|$.  Hence the implication $\|u(t) \| = O(\Phi(t)) \implies \|\theta(t) \| = o(\Phi(t)) $ is self-contained.
		This fact is key to establish the following ``inverse form'' of Theorem \ref{th:intro}, in the same spirit as \cite{CPZ} (see also \cite{Ska}), by using the triangle inequality.
	\end{Remark}
	\begin{Corollary}\label{coro:inv}
		Let $u_0\in L^2_\sigma(\R^2)$ and $f$ satisfy~
		\eqref{assu0:f}-\eqref{assu1:f}. Let $u$ be the Leray solution
		of the Navier--Stokes equations \eqref{NS} with $d=2$, and $v$ be the solution of the heat equation~\eqref{stokes-flow}, starting from~$u_0$.
		Suppose that
		\[
		\|u(t)\|=O(\Phi(t)),\quad\text{as $t\to+\infty$},
		\]
		with $\Phi$ satisfying conditions~\eqref{Ass-Phi0}-\eqref{Ass-Phi1}. Then, $\|v(t)\|=O(\Phi(t))$ as $t\to+\infty$. 
		
	\end{Corollary}

	An immediate consequence is the following equivalence between the two-sided estimates:
	under the hypotheses of Corollary \ref{coro:inv}, 
	\[ \|u(t)\|\approx \Phi(t) \iff
	\|v(t)\|\approx \Phi(t).\]

	\appendix
	
	\section{An illustration of~Theorem~\ref{th:intro}}

	In this Appendix, we provide an example motivated by the discussion following the statement of~Theorem~\ref{th:intro}.

	\begin{Example}
		\label{Ex:1}
		Let $\chi\in C^\infty_c(\R^2)$ be a non-radial function,
		identically equal to~$1$ in $\{\xi\in\R^2\colon|\xi|\le 1/4\}$, supported in the disc $\{\xi\in\R^2\colon|\xi|\le 1/2\}$.
		Let  $\widehat\omega_0(\xi)=|\ln|\xi||^{-1}\chi(\xi)$.
		Let $\widehat u_0(\xi)=|\xi|^{-2}(i\xi_2,-i\xi_1)\widehat\omega_0(\xi)$.
		Then, $u_0\in L^2_\sigma(\R^2)$, and $\omega_0$ is the corresponding initial vorticity, namely $\omega_0=\mbox{curl\,} u_0$. We choose $f\equiv0$.
		The non-radiality of~$\chi$ is just useful to avoid the somewhat trivial case $\P(u\cdot \nabla \omega)\equiv0$, corresponding to the situation where
		the solution of the Navier--Stokes equations boils down to the solution of the heat equation. 
		Let us prove that
		\begin{equation}
			\label{exadec}
			\|v(t)\|\approx (\ln(e+t))^{-1/2}.
		\end{equation}

		We apply Schonbek's Fourier-Splitting method~\cite{Sch85}. Performing an energy estimate on the heat equation~\eqref{stokes-flow} yields 
		\[
		\frac12 \frac{\dd}{\dd t}\|v(t)\|^2 + \frac1t \int_{|\xi|\ge(\sqrt t)^{-1}} |\widehat v(\xi,t)|^2\dd\xi
		\le
		-\int_{|\xi|\le(\sqrt t)^{-1}}|\xi|^2|\widehat v(\xi,t)|^2\dd\xi.
		\]
		From this, we deduce 
		\begin{equation}
			\label{fsm}
			\frac12 \frac{\dd}{\dd t}\|v(t)\|^2 
			+\frac1t \|v(t)\|^2
			\le 
			\frac{1}{t}\int_{|\xi|\le(\sqrt t)^{-1}}|\widehat v(\xi,t)|^2\dd\xi.
		\end{equation}
		But for $t\ge16$,
		\[
		\begin{split}
			\int_{|\xi|\le(\sqrt t)^{-1}}|\widehat v(\xi,t)|^2\dd\xi
			&=  \int_{|\xi|\le(\sqrt t)^{-1}}e^{-2t|\xi|^2}|\xi|^{-2}|\ln|\xi||^{-2}			\dd\xi\\
			&\approx \int_{(\sqrt{t})^{-1}}^{+\infty} \rho^{-1}|\ln\rho|^{-2}\dd\rho\\
			&\approx(\ln t)^{-1}.
		\end{split}
		\]
		So in particular, for all $t\ge 2$, $\|v(t)\|\gtrsim (\ln t)^{-1/2}$.
		But multiplying~\eqref{fsm} by~$t$ and integrating over $[2,t]$,
		we get, for all $t\ge2$, 
		\[
		t\|v(t)\|^2\lesssim 
		\int_e^t (\ln s)^{-1}\dd s
		\lesssim t(\ln t)^{-1}.
		\]
		These calculations imply~\eqref{exadec} and the application of Theorem~\ref{th:intro}
		shows that $\|u(t)\|$ has the same large-time behavior.
	\end{Example}

\end{document}